\RequirePackage{plautopatch}
\documentclass[dvipdfmx,12pt]{article}
\usepackage[top=25truemm,bottom=20truemm,left=20truemm,right=20truemm]{geometry}
\usepackage[noadjust]{cite}
\usepackage{hyperref}
\usepackage{xcolor}
\usepackage{amsthm}
\usepackage{amsmath,amssymb,amsfonts,amssymb,amscd} 
\usepackage[all]{xy}
\usepackage{mathtools}
\usepackage{cleveref}

\theoremstyle{plain}
\newtheorem{thm}{Theorem}[section]
\crefname{thm}{Theorem}{Theorem}

\crefname{cor}{Corollary}{Corollary}
\newtheorem{lem}[thm]{Lemma}
\crefname{lem}{Lemma}{Lemma}
\newtheorem{prop}[thm]{Proposition}
\crefname{prop}{Proposition}{Proposition}

\theoremstyle{definition}
\newtheorem{dfn}[thm]{Definition}
\crefname{dfn}{Definition}{Definition}
\newtheorem{ex}[thm]{Example}
\crefname{ex}{Example}{Example}
\newtheorem{rmk}[thm]{Remark}
\crefname{rmk}{Remark}{Remark}

\newcommand{\id}{\mathrm{id}}
\newcommand{\RR}{\mathbb{R}}
\newcommand{\CC}{\mathbb{C}}
\newcommand{\EE}{\mathbb{E}}
\newcommand{\HH}{\mathbb{H}}
\newcommand{\ZZ}{\mathbb{Z}}

\newcommand{\tr}{\mathrm{tr}\,}

\newcommand{\Isom}{\mathrm{Isom}}
\newcommand{\Fix}[1]{\mathrm{Fix}(#1)}

\newcommand{\Grp}{\mathrm{Grp}}
\newcommand{\Aut}{\mathrm{Aut}}
\newcommand{\Inn}{\mathrm{Inn}}
\newcommand{\SL}{\mathrm{SL}}
\newcommand{\PSL}{\mathrm{PSL}}

\newcommand{\Qdle}{\mathrm{Qdle}}
\newcommand{\Conj}{\mathrm{Conj}}
\newcommand{\As}{\mathrm{As}}

\newcommand{\address}[1]{\bigskip{\small\noindent #1 \par}}
\newcommand{\email}[1]{{\small\noindent\textit{Email address}: \texttt{#1} \par}}

\title{On the quandles of isometries of the hyperbolic 3-space}
\author{Ryoya Kai}
\date{}

\begin{document}
    \maketitle
    \begin{abstract}
  A quandle is an algebraic structure 
  whose axioms are related to the Reidemeister moves 
  used in knot theory.
  In this paper, 
  we investigate the conjugate quandle 
  of the orientation-preserving isometry group $\PSL(2, \CC)$
  of hyperbolic 3-space and its subquandles.
  We introduce a quandle, denoted by $Q(\Gamma, \gamma)$, 
  associated with a pair $(\Gamma, \gamma)$.
  Here, $\Gamma$ is a Kleinian group,
  and $\gamma$ is a non-trivial element of $\Gamma$.
  This construction can be regarded as a generalization 
  of knot quandles to hyperbolic knots.
  Moreover, for pairs $(\Gamma, \gamma)$ satisfying certain conditions,
  we construct the canonical map 
  from $Q(\Gamma, \gamma)$ to the conjugate quandle of $\PSL(2,\CC)$,
  which is an injective quandle homomorphism with a discrete image.
\end{abstract}

\section{Introduction}
A knot is a smooth embedding of $S^1$ into $S^3$.
A hyperbolic knot is a knot 
whose complement has 
a complete hyperbolic structure of finite volume.
The holonomy representations of hyperbolic structures
are crucial for studying
hyperbolic structures of knot complements
(or more generally 3-manifolds).
In particular, 
a complete hyperbolic structure of finite volume
is a topological invariant by the Mostow rigidity theorem.
A holonomy representation associated with
the complete hyperbolic structure of finite volume
is faithful, and its image is discrete in $\PSL(2, \CC)$.

A quandle is an algebraic structure defined independently 
by Joyce \cite{Joyce-1982-ClassifyingInvariantKnotsKnot}
and Matveev \cite{Matveev-1982-DistributiveGroupoidsKnotTheory}.
The axioms correspond to the Reidemeister moves of the diagrams of a knot.
The knot quandle $Q(K)$ is defined for a knot $K$
and provides some invariants of $K$
(see \cite{Kamada-2017-SurfaceKnots4Space}).
For a quandle $X$, an $X$-coloring of a knot $K$ is 
a quandle homomorphism from 
the knot quandle $Q(K)$ to the quandle $X$.
A quandle coloring 
can be regarded as a refinement 
of a group homomorphism from the knot group 
$G(K) = \pi_1(S^3 \setminus K)$.

Inoue \cite{Inoue-2010-QuandleHyperbolicVolume}
and Inoue-Kabaya \cite{Inoue-2014-QuandleHomologyComplexVolume}
have investigated
the relationship between hyperbolic volume (or complex volume) and quandle homology.
They showed in \cite{Inoue-2014-QuandleHomologyComplexVolume}
that the complex volume of the complement
of a hyperbolic knot $K$ 
can be calculated using a $\mathcal{P}$-coloring of $K$,
where $\mathcal{P}$ is a quandle 
composed of 
parabolic elements in $\PSL(2, \CC)$.
It can be observed that the image of a quandle coloring $Q(K) \to \mathcal{P}$ 
for a hyperbolic knot $K$ is discrete in $\PSL(2, \CC)$ as follows:
Since hyperbolic knots are prime,
the natural map $i: Q(K) \to \As(Q(K))$ is injective,
where $\As(X)$ is the associated group of $X$
\cite[Corollary 3.6]{Ryder-1996-AlgebraicConditionDetermineWhether}.
The group $\As(Q(K))$ is isomorphic to the knot group $G(K)$
\cite[Corollary 3.3]{Fenn-1992-RacksLinksCodimensionTwo}.
Since $K$ is hyperbolic, 
there exists an injective group homomorphism $\rho: G(K) \to \PSL(2, \CC)$
with a discrete image.
Consequently,
the image of $\rho \circ i$ is discrete.
This argument depends on $Q(K)$ being 
the knot quandle of a prime knot $K$.

We introduce a quandle $Q(\Gamma, \gamma)$
for a pair consisting of a Kleinian group $\Gamma$
and a non-trivial element $\gamma \in \Gamma$
(see \cref{KleinQdle}).
This quandle can be regarded as 
a generalization of the knot quandles for hyperbolic knots
(see \cref{hypKnotQdle}).
Additionally, we introduce the canonical map $Q(\Gamma, \gamma) \to \Conj(\PSL(2,\CC))$
(see \cref{canonicalmap}).
The main theorem of this paper 
generalizes the result concerning 
the discreteness of hyperbolic knot quandles 
to the quandles $Q(\Gamma, \gamma)$ as follows:
\begin{thm}\label{mainthm}
  Let $\Gamma$ be a Kleinian group  
  and $\gamma$ a non-trivial element in $\Gamma$.
  Then the canonical map 
  $Q(\Gamma, \gamma) \to \Conj(\PSL(2, \CC))$
  is an injective quandle homomorphism.
  Moreover, if either
  \begin{enumerate}
    \item $\Gamma$ is isomorphic to the fundamental group of 
    a complete hyperbolic $3$-orbifold of finite volume
    and $\gamma$ is parabolic, or
    \item 
    $\gamma$ is not parabolic
    and the cardinality of $C_\Gamma(\gamma)$ is infinite,
  \end{enumerate}
  then its image is discrete in $\Conj(\PSL(2, \CC))$
  and is contained in a connected component of 
  $\Conj(\PSL(2, \CC))$.
\end{thm}

\cref{mainthm} indicates that 
the quandle $Q(\Gamma, \gamma)$ 
serves as an analogous object
in the theory of quandles associated with Kleinian groups.
The key idea of the proof 
involves 
employing a quandle triplet 
associated with the connected component of $\Conj(\PSL(2, \CC))$.

This paper is organized as follows.
In Section \ref{Quandle}, we review quandles.
In Section \ref{Hyp3sp},
we review the hyperbolic 3-space
and Kleinian groups.
In Section \ref{SL2},
we study the structure of the conjugate quandle
$\Conj(\SL(2,\CC))$ (see \cref{SLmain}).
In Section \ref{PSL2},
we study the structure of the conjugate quandle
$\Conj(\PSL(2,\CC))$.
Using the result of Section \ref{SL2} and 
the natural projection $\pi: \SL(2, \CC) \to \PSL(2, \CC)$,
we give the decomposition of $\Conj(\PSL(2,\CC))$ into connected components
and a presentation of connected subquandles 
by a quandle triplet (see \cref{PSLmain}).
In Section \ref{Klein},
we define the quandle $Q(\Gamma, \gamma)$
and provide the proof of the main theorem.

    \section{Quandle}\label{Quandle}

\begin{dfn}[\cite{Joyce-1982-ClassifyingInvariantKnotsKnot,Matveev-1982-DistributiveGroupoidsKnotTheory}]
  A \emph{quandle} is a non-empty set 
  $X$ with a binary operation
  $\lhd: X \times X \to X$ satisfying 
  the following three axioms:
  \begin{enumerate}
    \item for any $x \in X$, $x \lhd x = x$,
    \item for any $y \in X$, the map $s_y: X \to X$ defined by $s_y(x)=x \lhd y$ is a bijection,
    \item for any $x, y, z \in X$, $(x \lhd y) \lhd z = (x \lhd z) \lhd (y \lhd z)$.
  \end{enumerate}
\end{dfn}

The bijection $s_y: X \to X$ is called 
the \emph{point symmetry} at $y \in X$.
A quandle $Y$ is a \emph{subquandle of $X$}
if $Y$ is a subset of $X$ and is a quandle with 
the binary operation that is the restriction of that of $X$.

Let $X$ and $Y$ be quandles. 
A map $f: X \to Y$ is a \emph{quandle homomorphism}
if $f(x \lhd x') = f(x) \lhd f(x')$ for any $x, x' \in X$.
A quandle homomorphism $f: X \to Y$ is a \emph{quandle isomorphism}
if there exists a quandle homomorphism $g: Y \to X$
such that $g \circ f = \id_X$ and $f \circ g = \id_Y$,
and in this case, we say $X$ and $Y$ are \emph{isomorphic}.
A map $f: X \to Y$ is a quandle isomorphism 
if and only if  $f$ is a bijective quandle homomorphism.
Note that the point symmetry $s_y$ for any $y \in X$ 
is a quandle isomorphism.

We denote the set 
of quandle isomorphisms from $X$ to $X$
by $\Aut^\Qdle{(X)}$.
We define a multiplication on $\Aut^\Qdle{(X)}$ 
by $fg:=g \circ f$.
Then $\Aut^\Qdle{(X)}$ is a group acting 
on $X$ from the right.
A quandle is \emph{homogeneous} 
if the action of $\Aut^\Qdle{(X)}$ is transitive.
The subgroup of $\Aut^\Qdle{(X)}$ 
generated by all point symmetries is 
called the \emph{inner automorphism group} 
and is denoted by $\Inn^\Qdle{(X)}$.
We denote the action of $\Inn^\Qdle{(X)}$
by $x \lhd g$ for $x \in X$ and $g \in \Inn^\Qdle{(X)}$.
A quandle is \emph{connected} 
if the action of $\Inn^\Qdle{(X)}$ is transitive.
\emph{A connected component} of $X$ is
an orbit of the action of the inner automorphism group on the quandle.
$\Inn^\Qdle(X)$ is a normal subgroup of $\Aut^\Qdle(X)$.
Note that a connected component is a subquandle.
A connected component is not necessarily a connected quandle.

\begin{ex}\label{ConjQdle}
  Let $G$ be a group. 
  Let $\iota_g: G \to G$ be
  the inner automorphism of $g \in G$
  defined by $\iota_g(x) = gxg^{-1}$.
  We define a binary operation on $G$ 
  by $x \lhd y := \iota_{y^{-1}}(x)$.
  Then $\Conj(G):=(G, \lhd)$ is a quandle 
  called the \emph{conjugate quandle} of $G$.

  Let $G$ be a group
  and $H$ a subset of $G$ such that 
  $\iota_{h^{-1}}(H) \subset H$ for any $h \in H$.
  We define a binary operation on $H$ by $x \lhd y := \iota_{y^{-1}}(x)$.
  Then $\Conj(H) := (H, \lhd)$ is
  a subquandle of $\Conj(G)$.
  The inner automorphism group $\Inn^\Qdle(\Conj(H))$
  is isomorphic to the inner automorphism group 
  $\Inn^\Grp(\langle H \rangle_\Grp)$,
  where $\langle H \rangle_\Grp$ is  the subgroup generated by $H$
  and $\Inn^\Grp(G)$ is the inner automorphism group of a group $G$.
  Indeed, the map $s_h \mapsto \iota_{h}$ can be
  extended to a group isomorphism.
\end{ex}

\begin{ex}\label{QuandleTriplet}
  Let $G$ be a group,
  $\sigma$ a group automorphism of $G$,
  and $H$ a subgroup of 
  the centralizer $C_G(\sigma) = \{g \in G \mid \sigma(g) = g\}$
  of $\sigma$.
  Then the triplet $(G,H,\sigma)$ is called a \emph{quandle triplet}.
  We can define a binary operation on $H \backslash G$ 
  by $Hx \lhd Hy := H\sigma(xy^{-1})y$.
  Then $Q(G, H, \sigma) := (H \backslash G, \lhd)$ 
  is a quandle.
  Note that the group $G$ acts on $Q(G, H, \sigma)$ 
  from the right
  by $Hx \cdot g:= H xg$ as quandle automorphisms,
  and hence $Q(G, H, \sigma)$ is a homogeneous quandle.
\end{ex}

\begin{prop}[{
  \cite[Theorem 7.1.]{Joyce-1982-ClassifyingInvariantKnotsKnot}}]
  \label{homogeneousQuandle}
  Let $X$ be a homogeneous quandle
  and $G$ a normal subgroup of $\Aut^\Qdle{(X)}$
  acting transitively on $X$.
  Take a base point $x_0 \in X$.
  Suppose that $H$ is the stabilizer subgroup 
  at $x_0$ of the action of $G$ on $X$.
  Define
  a group automorphism $\sigma: G \to G$
  by $\sigma(f) = s_{x_0}^{-1} f s_{x_0}$,
  where $s_{x_0}$ is the point symmetry at $x_0$.
  Then the following hold.
  \begin{enumerate}
    \item $(G,H,\sigma)$ is a quandle triplet.
    \item The map $\phi:H \backslash G \to X$
    defined by $\phi(Hg) = x_0 \cdot g$ 
    is a quandle isomorphism.
  \end{enumerate}
\end{prop}

\begin{lem}\label{injective_lemma}
  Let $(G, H, \sigma)$ and $(G', H', \sigma')$ 
  be quandle triplets 
  and 
  $\rho: G' \to G$ a group homomorphism.
  Suppose that $\sigma \circ \rho = \rho \circ \sigma'$
  and $\rho(H') \subset H$. 
  Then the following hold.
  \begin{enumerate}
    \item\label{inj1} 
    The map 
    $\overline{\rho}: Q(G', H', \sigma') 
    \to Q(G, H, \sigma)$
    defined by $\overline{\rho}(H'g) = H\rho(g)$
    is a quandle homomorphism.
    
    \item\label{inj2} 
    If $H' = \rho^{-1}(H)$, 
    then the quandle homomorphism 
    $\overline{\rho}$ is injective.
  \end{enumerate}
\end{lem}

\begin{proof}
  (\ref{inj1})
  First, we show the well-definedness of 
  the map $\overline{\rho}$.
  Let $a, b \in G$ such that $ab^{-1} \in H'$.
  Then $\rho(a)\rho(b)^{-1} = \rho(ab^{-1}) 
  \in \rho(H') \subset H$.
  Thus, $H'a = H'b$ implies 
  $\overline{\rho}(H'a) = \overline{\rho}(H'b)$,
  and 
  hence $\overline{\rho}$ is well-defined.
  Next, we show 
  $\overline{\rho}$ is a quandle homomorphism.
  For any $a, b \in G'$,
  \begin{align*}
    \overline{\rho}(H'a \lhd H'b)
    = \overline{\rho}(H'\sigma'(ab^{-1})b)
    = H \rho(\sigma'(ab^{-1})b)
    = H \rho (\sigma'(ab^{-1})) \rho(b),
  \end{align*}
  using 
  the definition of the binary operation of
  $Q(G', H', \sigma')$ at the first equation,
  the definition of the map $\overline{\rho}$
  in the second equation
  and 
  that $\rho$ is a group homomorphism
  in the third equation.
  Since $\rho(\sigma'(ab^{-1})) 
  = \sigma(\rho(ab^{-1}))$
  from the assumption, 
  $H \sigma (\rho(ab^{-1})) \rho(b)
  = H \sigma (\rho(a)\rho(b)^{-1}) \rho(b)$.
  Hence,
  \begin{align*}    
    H \sigma (\rho(a)\rho(b)^{-1}) \rho(b)
    = H \rho(a) \lhd H\rho(b)
    = \overline{\rho}(H'a) \lhd \overline{\rho}(H'b),
  \end{align*}
  using 
  the definition of the binary operation of
  $Q(G, H, \sigma)$ in the first equation
  and 
  the definition of the map $\overline{\rho}$
  in the second equation.
  Hence,
  $\overline{\rho}(H'a \lhd H'b) 
  = \overline{\rho}(H'a) \lhd \overline{\rho}(H'b)$.
  Therefore,
  $\overline{\rho}$ is a quandle homomorphism.

  (\ref{inj2})
  Suppose that 
  $\overline{\rho}(H'a) = \overline{\rho}(H'b)$ 
  for $a,b \in G'$.
  Since $\rho(ab^{-1}) \in H$ and $H' = \rho^{-1}(H)$ by assumption,
  we have $ab^{-1} \in \rho^{-1}(H) = H'$.
  Therefore $H'a = H'b$.
\end{proof}


    \section{The hyperbolic 3-space}\label{Hyp3sp}
The hyperbolic 3-space $\HH^3$ is 
a simply connected Riemannian manifold
with constant negative curvature $-1$.
In this paper, we employ the upper half-space model $U^3$,
which is the upper half-space in $\RR^3$ with the hyperbolic metric.
It is regarded as a subset of the set of quaternions:
\begin{equation*}
  U^3 = \{z+tj \mid z \in \CC,\,t>0 \}.
\end{equation*}
The ideal boundary $\partial \HH^3$ of $\HH^3$ is 
identified with the Riemann sphere $\CC P^1$.
We denote the compactification of $\HH^3$ 
by $\overline{\HH}^3 := \HH^3 \cup \partial \HH^3$.
The geodesics of $\HH^3$ are
the semicircles, or the half-lines perpendicular to the complex plane.
For each pair $(x,y)$ of points in $\partial \HH^3$,
there exists a unique geodesic $l$ of $\HH^3$ 
such that the set $\partial_\infty l$ of endpoints is equal to $\{x, y\}$.

For each $A =
\left(\begin{matrix}
  a & b \\ c & d
\end{matrix}\right)
\in \SL(2,\CC)$, 
we can define the orientation-preserving isometry 
$f_A: \HH^3 \to \HH^3$ by $f_A(q) = (aq+b)(cq+d)^{-1}$.
Then the map $\pi: \SL(2,\CC) \to \Isom^+(\HH^3)$ 
defined by $\pi(A) = f_A$ 
is a surjective group homomorphism 
with $\ker{(\pi)} = \{\pm I_2\}$.
Thus the orientation-preserving isometry group of $\HH^3$
is isomorphic to $\PSL(2,\CC) := \SL(2,\CC)/\{\pm I_2\}$.
$\PSL(2,\CC)$ naturally acts on $\partial \HH^3 = \CC P^1$
as linear fractional transformations.

A non-identity orientation-preserving isometry
$f \in \PSL(2,\CC)$
is classified by the trace $\tr(A) \in \CC$ for 
$A \in \SL(2,\CC)$ with $f_A=f$ as follows:
$f$ is \emph{elliptic} (resp. \emph{parabolic} and \emph{loxodromic})
if $\tr(A) \in (-2,2)$
(resp. $\tr(A) \in \{ \pm2 \}$ and $\tr(A) \in \CC \setminus [-2,2]$). 
It is also classified by the fixed points of the action on $\overline{\HH}^3$
as follows:
$f$ is elliptic (resp. parabolic and loxodromic)
if and only if $\Fix{f_A}$ is a geodesic in $\HH^3$
(resp. $\Fix{f_A}$ is a point in $\partial \HH^3$
and $\Fix{f_A}$ is a two-point set in $\partial \HH^3$).
Note that if $f \in \PSL(2, \CC)$
is not parabolic,
then there uniquely exists a geodesic $l$ 
in $\HH^3$ such that 
the restriction of $f$ to 
$\partial_\infty l $ is the identity map.
We denote some 
standard linear fractional transformations by
\begin{align*}
  S_\lambda(z) &= \lambda z & \text{for } \lambda \in \CC^*,\\
  S'_\lambda(z) &= -\frac{\lambda}{z} & \text{for } \lambda \in \CC^*,\\
  T_v(z) &= z + v & \text{for } v \in \CC.\\
\end{align*}
$S_\lambda$ and $T_v$ are the images of 
the Jordan normal forms of matrices in $\SL(2, \CC)$
by the projection.
In particular, 
$f \in \PSL(2, \CC)$
is elliptic (resp. parabolic and loxodromic)
if and only if 
$f$ is conjugate to $S_\lambda$
with $|\lambda|=1$ and $\lambda \neq 1$
(resp. to $T_1$ 
and to $S_\lambda$
with $|\lambda| \neq 1$).

A subgroup $G < \PSL(2, \CC)$ is \emph{elementary}
if there exists $ x \in \overline{\HH}^3$ such that 
the orbit $G \cdot x$ is finite 
(cf. \cite[\S 5.5.]{Ratcliffe-2019-FoundationsHyperbolicManifolds}).
An elementary group $G$
belongs to one of the following types:
\begin{enumerate}
  \item $G$ is said to be of \emph{elliptic type}
  if there exists $x \in \HH^3$ such that 
  the cardinality of the orbit $G \cdot x$ is finite.
  \item $G$ is said to be of \emph{parabolic type}
  if there exists $x \in \partial \HH^3$ such that 
  the cardinality of the orbit $G \cdot x$ is finite
  and for any $y \in \overline{\HH}^3$, 
  the cardinality of the orbit $G \cdot y$ is infinite.
  \item $G$ is said to be of \emph{hyperbolic type}
  if $G$ is neither elliptic type nor parabolic type.
\end{enumerate}
The discrete elementary groups are characterized as follows.
\begin{prop}[{\cite[Theorem 5.5.2.]{Ratcliffe-2019-FoundationsHyperbolicManifolds}}]\label{discElemEllipt}
  A subgroup $\Gamma$ of $\PSL(2, \CC)$ is an elementary discrete subgroup
  of elliptic type
  if and only if $\Gamma$ is finite.
\end{prop}

\begin{prop}[{\cite[Theorem 5.5.5.]{Ratcliffe-2019-FoundationsHyperbolicManifolds}}]\label{discElemPara}
  A subgroup $\Gamma$ of $\PSL(2, \CC)$ is an elementary discrete subgroup
  of parabolic type
  if and only if $\Gamma$ is conjugate in $\PSL(2, \CC)$
  to a discrete subgroup of $\Isom^+(\EE^2)$.
\end{prop}

\begin{prop}[{\cite[Theorem 5.5.8.]{Ratcliffe-2019-FoundationsHyperbolicManifolds}}]\label{discElemHyp}
  A subgroup $\Gamma$ of $\PSL(2, \CC)$ is an elementary discrete subgroup
  of hyperbolic type
  if and only if $\Gamma$ contains an infinite cyclic subgroup
  of finite index which is generated by a loxodromic element.
\end{prop}

    \section{The conjugate quandle of $\SL(2, \CC)$}\label{SL2}

In this section, we study the structure 
of the conjugate quandle $\Conj(\SL(2, \CC))$.
For each $t \in \CC$, we define 
the subset $\widetilde{X}_t$ of $\SL(2, \CC)$,
the complex number $\lambda_t$,
the base point $A_t \in \widetilde{X}_t$,
and the subgroup $\widetilde{H}_t$ of $\SL(2, \CC)$ by 
\begin{align*}
  \widetilde{X}_t 
    &= \{A \in \SL(2, \CC) \setminus \{\pm I_2\} \mid \tr A = t\},\\
  \lambda_{t} 
    &=\begin{cases*}
    e^{i \theta}
    & if $t = 2\cos \theta \in (-2, 2)$, $\theta \in (-\pi, \pi)$,
    \\
    \varepsilon
    & if $t = 2\varepsilon \in \{\pm 2\}$, $\varepsilon \in \{\pm 1\}$,
    \\
    \lambda
    & if $t = \lambda + \lambda^{-1} \in \CC \setminus [-2,2]$, 
    $\lambda \in \CC^*, |\lambda|>1$,
  \end{cases*}\\
  A_{t} 
    &=\begin{cases*}
      \begin{pmatrix}
        \lambda_t & 1 \\
        0 & \lambda_t^{-1} 
      \end{pmatrix}
      & if $t \in \{\pm 2\}$,
      \\
      \begin{pmatrix}
        \lambda_t & 0 \\
        0 & \lambda_t^{-1} 
      \end{pmatrix}
      & if $t \in \CC \setminus \{\pm 2\}$,\\
    \end{cases*}\\
  \widetilde{H}_t 
    &= \begin{cases*}
      \left\{
      \begin{pmatrix}
        \varepsilon & \mu \\
        0 & \varepsilon
      \end{pmatrix}
      \in \SL(2, \CC)
    \,\middle|\,
    \varepsilon \in \{\pm 1\},\, \mu \in \CC
    \right\}
    & if $t=\pm 2$,\\
    \left\{
      \begin{pmatrix}
        \mu & 0 \\
        0 & \mu^{-1}
      \end{pmatrix}
      \in \SL(2, \CC)
    \,\middle|\,
    \mu \in \CC^*
    \right\}
    & if $t \in \CC \setminus \{\pm 2\}$.
    \end{cases*}
\end{align*}
The group $\SL(2, \CC)$ acts on $\widetilde{X}_t$ from the right
by $A \cdot g := g^{-1} A g$ for $A \in \widetilde{X}_t$ and $g \in \SL(2,\CC)$.
The action is transitive
since the Jordan normal form of $A \in \widetilde{X}_t$
is the base point $A_t$.
Note that $A \in \widetilde{H}_t$ 
if and only if the matrices $A$ and $A_t$ commute,
and hence
$\widetilde{H}_{t}$ is the stabilizer subgroup
at $A_t$ under the action of $\SL(2, \CC)$.

$\widetilde{X}_{t}$ is a topological space
with the relative topology from $\SL(2, \CC)$.

\begin{lem}\label{X_tIsCpxMfd}
  For each $t \in \CC$,
  the subset $\widetilde{X}_t$ is 
  a 2-dimensional complex submanifold of $\CC^4$.
\end{lem}
\begin{proof}
  Consider the function $F:\CC^4 \to \CC^2$ 
  defined by 
  $F(a,b,c,d)=(ad-bc, a+d)$.
  Then $\widetilde{X}_t$ is equal to $F^{-1}(1, t)$.
  The Jacobian matrix $dF_x$ at $x = (a,b,c,d) \in \widetilde{X}_t$ 
  satisfies
  \begin{equation*}
    dF_x = \begin{pmatrix}
      d & -c & -b & a \\
      1 & 0 & 0 & 1
    \end{pmatrix}.
  \end{equation*}
  Since $(a,b,c,d) \neq \pm (1,0,0,1)$,
  the rank of $dF_x$ is equal to $2$.
  Therefore, 
  $\widetilde{X}_t$ is a 2-dimensional 
  complex submanifold of $\CC^4$.
\end{proof}

Define the equivalence relation 
on $\CC^2 \setminus \{(0,0)\}$ by $v \sim -v$
and denote the quotient space 
by $\mathcal{P}$.
For $(\alpha, \beta) \in \CC^2 \setminus \{(0,0)\}$,
define the matrix $A_{(\alpha, \beta)} \in \widetilde{X}_2$ by
\begin{equation*}
  A_{(\alpha, \beta)} := 
  \begin{pmatrix}
    1+\alpha \beta & \alpha^2 \\
    -\beta^2 & 1- \alpha \beta
  \end{pmatrix}.
\end{equation*}

\begin{lem}[cf. \cite{Inoue-2014-QuandleHomologyComplexVolume,Rubinsztein-2007-TOPOLOGICALQUANDLESINVARIANTSLINKS}]\label{coodinate_parabolic}
  The map $F_2: \mathcal{P} \to \widetilde{X}_2$ 
  defined by $F_2([(\alpha, \beta)])=A_{(\alpha, \beta)}$
  is a bijection.
\end{lem}

\begin{proof}
  For any elements $B \in \widetilde{X}_2$,
  there exists $A \in \SL(2, \CC)$ such that 
  $B=A^{-1} A_2 A$.
  Any matrix $A=\begin{pmatrix}
    a & b\\
    c & d
  \end{pmatrix} \in \SL(2, \CC)$
  satisfies
  \begin{equation*}
    A^{-1} A_2 A = \begin{pmatrix}
      1+cd & d^2 \\
      -c^2 & 1-cd
    \end{pmatrix}
    =A_{(d,c)}.
  \end{equation*}
  For $(\alpha, \beta) \in \CC^2 \setminus \{(0,0)\}$,
  put
  \begin{equation*}
    A = \begin{cases*}
      \begin{pmatrix}
        \alpha^{-1} & 0 \\
        \beta & \alpha
      \end{pmatrix} & if $\alpha \neq 0$,\\
      \begin{pmatrix}
        0 & -\beta^{-1} \\
        \beta & \alpha
      \end{pmatrix} & if $\alpha = 0$,
    \end{cases*}
  \end{equation*}
  then we have $A \in \SL(2, \CC)$
  and $A^{-1} A_2 A = A_{(\alpha, \beta)}$,
  and hence $F_2$ is surjective.
  If 
  $A_{(\alpha,\beta)} = A_{(\alpha',\beta')}$,
  then $\alpha^2 = \alpha'^2$, 
  $\beta^2 = \beta'^2$ and 
  $\alpha \beta = \alpha' \beta'$.
  Hence $(\alpha, \beta) = \pm (\alpha', \beta')$,
  and $F_2$ is injective.
\end{proof}

Let $t \in \CC \setminus \{\pm 2\}$ and 
$\mathcal{M}:=\{(\alpha, \beta, \gamma) 
\in \CC^3 \mid \alpha(1-\alpha)+\beta\gamma=0 \}$.
For $(\alpha, \beta,\gamma) \in \mathcal{M}$,
define a matrix 
$A^{\lambda_t}_{(\alpha, \beta,\gamma)} \in \widetilde{X}_t$ by
\begin{equation*}
  A_{(\alpha, \beta,\gamma)}^{\lambda_t} := 
  \begin{pmatrix}
    \alpha\lambda_t + (1-\alpha) \lambda_t{-1} & ({\lambda_t} - {\lambda_t}^{-1})\beta \\
    -({\lambda_t} - {\lambda_t}^{-1})\gamma & (1-\alpha) {\lambda_t} + \alpha {\lambda_t}^{-1}
  \end{pmatrix}.
\end{equation*}

\begin{lem}
  For any $t \in \CC \setminus \{ \pm 2\}$,
  the map $F_t: \mathcal{M} \to \widetilde{X}_t$ 
  defined by $F_t((\alpha, \beta, \gamma)) = A^{\lambda_t}_{(\alpha, \beta,\gamma)}$
  is a bijection.
\end{lem}

\begin{proof}
  For any element $B \in \widetilde{X}_t$,
  there exists $A \in \SL(2, \CC)$ such that 
  $B=A^{-1} A_t A$.
  Any matrix $A=\begin{pmatrix}
    a & b\\
    c & d
  \end{pmatrix} \in \SL(2, \CC)$
  satisfies
  \begin{equation*}
    A^{-1} A_t A = \begin{pmatrix}
      \lambda_t ad + \lambda_t^{-1} (1- ad) & (\lambda_t - \lambda_t^{-1}) bd \\
      -(\lambda_t - \lambda_t^{-1}) ac & \lambda_t^{-1}ad + \lambda_t (1 - ad)
    \end{pmatrix}
    =A^{\lambda_t}_{(ad, bd, ac)}.
  \end{equation*}
  For $(\alpha, \beta, \gamma) \in \mathcal{M}$,
  let
  \begin{equation*}
    A = \begin{cases*}
      \begin{pmatrix}
        \alpha & \beta \\
        \frac{\gamma}{\alpha} & 1
      \end{pmatrix} 
      & if $\alpha \neq 0$,\\
      \begin{pmatrix}
        1 & -\gamma^{-1} \\
        \gamma & 0
      \end{pmatrix}
      & if $\alpha = 0$ and $\beta=0$,\\
      \begin{pmatrix}
        0 &  \beta \\
        -\beta^{-1} & 1
      \end{pmatrix}
      & if $\alpha = 0$ and $\gamma=0$,
    \end{cases*}
  \end{equation*}
  then we have $A \in \SL(2, \CC)$
  and $A^{-1} A_t A = A^{\lambda_t}_{(\alpha, \beta,\gamma)}$,
  and hence $F_t$ is surjective.
  We can see that
  $A^{\lambda_t}_{(\alpha,\beta,\gamma)} 
  = A^{\lambda_t}_{(\alpha',\beta', \gamma')}$
  implies $(\alpha, \beta, \gamma) = (\alpha', \beta', \gamma')$.
  Therefore, $F_t$ is injective.
\end{proof}

\begin{prop}\label{homogSpSL}
  The map $F: \widetilde{H}_t \backslash \SL(2, \CC) \to \widetilde{X}_t$
  defined by $F(\widetilde{H}_t g) = A_t \cdot g$ for $g \in \SL(2, \CC)$
  is a homeomorphism.
\end{prop}

For the proof of \cref{homogSpSL},
we use the following fact about the topology of 
a homogeneous space of a topological group.

\begin{lem}[{\cite[Section 1.2, Theorem 1]{Godement-2017-IntroductionTheoryLieGroups}}]\label{homogHomeo}
  Let $G$ be a 
  topological group
  which is $\sigma$-compact, locally compact
  and Hausdorff.
  Let $X$ be locally compact Hausdorff space
  on which $G$ acts continuously and transitively.
  Let $H$ be the stabilizer subgroup 
  at a point $a \in X$.
  Then the map $f: G/H \to X$ defined 
  by $f(gH) = g \cdot a$ is a homeomorphism.
\end{lem}

\begin{proof}[Proof of \cref{homogSpSL}]
  Apply \cref{homogHomeo} to 
  the transitive action of $\SL(2,\CC)$ on $\widetilde{X}_t$.
\end{proof}

Since conjugation preserves the trace, 
we can define a quandle structure on $\widetilde{X}_t$
as the subquandle of $\Conj(\SL(2,\CC))$
and denote the quandle by $\Conj(\widetilde{X}_t)$.
Note that the map $\phi: \SL(2, \CC) \to \SL(2, \CC)$ defined by
$\phi(A) = -A$ induces a quandle isomorphism
$\Conj(\widetilde{X}_t) \to \Conj(\widetilde{X}_{-t})$ for each $t \in  \CC$.

\begin{prop}\label{connectedSL}
  Let $t \in \CC$.
  \begin{enumerate}
    \item\label{connSL1} 
    The subset $\widetilde{X}_{t}$ generates $\SL(2,\CC)$ as a group.
    \item\label{connSL2} 
    The inner automorphism group $\Inn^\Qdle(\Conj(\widetilde{X}_{t}))$
    is isomorphic to $\PSL(2, \CC)$.
    \item\label{connSL3} 
    The quandle $\Conj(\widetilde{X}_t)$ is a connected quandle.
  \end{enumerate}
\end{prop}

\begin{proof}
  (\ref{connSL1})
  For $z \in \CC$,
  we define two matrices
  $U_z=
    \begin{pmatrix}
      1 & z\\
      0 & 1
    \end{pmatrix}$,
  $L_z=
  \begin{pmatrix}
    1 & 0\\
    z & 1
  \end{pmatrix}$
  in $\widetilde{X}_2$.
  The matrices $U_z$ and $L_z$ 
  are contained in the subgroup
  $\langle \widetilde{X}_{-2} \rangle_\Grp$
  of $\SL(2, \CC)$
  generated by $\widetilde{X}_{-2}$
  since the following equations hold:
  \begin{align*}
    U_z &= \left(
      \begin{matrix}
        -1 & -z+1 \\
        0 & -1
      \end{matrix}
    \right)
    \left(
      \begin{matrix}
        -1 & -1 \\
        0 & -1
      \end{matrix}
    \right),\\
    L_z &= 
    \left(
      \begin{matrix}
        -1 & -1 \\
        0 & -1
      \end{matrix}
    \right)
    \left(
      \begin{matrix}
        -1 & 0 \\
        -z+1 & -1
      \end{matrix}
    \right).
  \end{align*}
  For any $t \in \CC \setminus \{\pm 2\}$,
  we have
  \begin{align*}
    U_z &= \left(
      \begin{matrix}
        \lambda_t & \lambda_t^{-1}z \\
        0 & \lambda_t^{-1}
      \end{matrix}
    \right)
    \left(
      \begin{matrix}
        \lambda_t^{-1} & 0 \\
        0 & \lambda_t
      \end{matrix}
    \right),\\
    L_z &= 
    \left(
      \begin{matrix}
        \lambda_t^{-1} & 0 \\
        0 & \lambda_t
      \end{matrix}
    \right)
    \left(
      \begin{matrix}
        \lambda_t & 0 \\
        \lambda_t^{-1}z & \lambda_t^{-1}
      \end{matrix}
    \right),
  \end{align*}
  and hence 
  the matrices $U_z$ and $L_z$ 
  are contained in the subgroup
  $\langle \widetilde{X}_{t} \rangle_\Grp$
  of $\SL(2, \CC)$
  generated by $\widetilde{X}_{t}$.
  The set $\{U_z, L_z \mid z \in \CC\}$ 
  generates $\SL(2, \CC)$
  since
  any $A = \begin{pmatrix}
    a & b \\
    c & d
  \end{pmatrix} \in \SL(2, \CC)$ satisfies
  \begin{equation*}
    A= \begin{cases*}
      L_{\frac{d-1}{b}}\, U_b \, L_{\frac{a-1}{b}} & if $b \neq 0$, \\
      U_{\frac{a-1}{c}}\, L_c\, U_{\frac{1-a}{ac}} & if $b = 0$, $c \neq 0$,\\
      U_{-a(a-1)}\, L_{-a^{-1}}\, U_{a-1}\, L_1 & if $b=c=0$.
    \end{cases*}
  \end{equation*}
  Therefore
  $\widetilde{X}_t$ generates $\SL(2, \CC)$
  for any $t \in \CC$.

  (\ref{connSL2})
  Since $\langle \widetilde{X}_t \rangle_{\Grp}$ is equal to $\SL(2, \CC)$
  by (\ref{connSL1}), 
  the inner automorphism group $\Inn^\Qdle(\Conj(\widetilde{X}_t))$
  is isomorphic to the inner automorphism group
  $\Inn^\Grp(\langle \widetilde{X}_t \rangle_{\Grp}) = \Inn^\Grp(\SL(2,\CC))$.
  Then we have 
  \begin{equation*}
    \Inn^\Grp(\SL(2,\CC))=\SL(2,\CC) / Z(\SL(2,\CC)) \cong \PSL(2,\CC),
  \end{equation*}
  where $Z(G)$ is the center of $G$.
  Therefore, $\Inn^\Qdle(\Conj(\widetilde{X}_t))$ is isomorphic to $\PSL(2,\CC)$.

  (\ref{connSL3})
  The surjective group homomorphism
  $s: \SL(2, \CC) \to \Inn^\Qdle{(\Conj(\widetilde{X}_t))}$
  satisfies $A \cdot g = A \lhd s(g)$
  for $A \in \Conj(\widetilde{X}_t)$ and $g \in \SL(2, \CC)$.
  Since the action of $\SL(2, \CC)$ is transitive,
  $\Inn^\Qdle{(\Conj(\widetilde{X}_t))}$ acts transitively,
  which completes proof.
\end{proof}

We define an automorphism 
$\widetilde{\sigma}_t \in \Aut(\SL(2, \CC))$ by the inner automorphism
$\widetilde{\sigma}_{t} = \iota_{A_t^{-1}}$ of $A_t^{-1}$.
The restriction of the automorphism $\widetilde{\sigma}_t$ 
to the quandle $\widetilde{X}_t$ 
is equal to the point symmetry of $\widetilde{X}_t$
at the base point $A_t$.

\begin{prop}\label{homogPresXt}
  The quandle $\Conj(\widetilde{X}_t)$ is isomorphic to 
  $Q(\SL(2, \CC), \widetilde{H}_t, \widetilde{\sigma}_t)$.
\end{prop}

\begin{proof}
  Apply \cref{homogeneousQuandle} to
  the action of $\SL(2,\CC)$ on $\Conj(\widetilde{X}_t)$.
\end{proof}

\begin{thm}\label{SLmain}
  The conjugate quandle $\Conj(\SL(2, \CC))$ 
  of $\SL(2,\CC)$ 
  is decomposed
  into connected components as follows:
  \begin{equation*}
    \Conj(\SL(2, \CC)) = \{I_2\} \sqcup \{-I_2\} \sqcup 
    \left(\bigsqcup_{t \in \CC} \Conj(\widetilde{X}_t) \right).
  \end{equation*}
  Each connected component $\Conj(\widetilde{X}_t)$ 
  is
  \begin{enumerate}
    \item\label{SLmainSpace} 
    a $2$-dimensional complex submanifold of $\CC^4$
    that is homeomorphic to $\widetilde{H}_t \backslash \SL(2, \CC)$, and
    \item\label{SLmainQdle}
    a connected subquandle of $\Conj(\SL(2, \CC))$
    that is isomorphic to $Q(\SL(2, \CC), \widetilde{H}_t, \widetilde{\sigma}_t)$.
  \end{enumerate}
\end{thm}

\begin{proof}
  For any $A \in \SL(2, \CC) \setminus \{\pm I_2\}$,
  there uniquely exists $t \in \CC$
  such that $A \in \widetilde{X}_t$
  by the definition of $\widetilde{X}_t$.
  Since conjugation preserves the trace
  and \cref{connectedSL},
  $\widetilde{X}_t$ is a connected component 
  of $\Conj(\SL(2,\CC))$.
  Since the center $Z(\SL(2, \CC))$ of $\SL(2,\CC)$ 
  is equal to $\{\pm I_2\}$,
  each element in $Z(\SL(2,\CC))$ is a connected component. 
  Hence, we have the decomposition.
  
  The topological space
  $\Conj(\widetilde{X}_t)=\widetilde{X}_t$
  is a $2$-dimensional complex submanifold of $\CC^4$
  by \cref{X_tIsCpxMfd}
  and homeomorphic to $\widetilde{H}_t \backslash \SL(2, \CC)$
  by \cref{homogSpSL}.
  Hence, we have (\ref{SLmainSpace}).
  $\Conj(\widetilde{X}_t)$ is a connected quandle by \cref{connectedSL}
  and isomorphic to $Q(\SL(2, \CC), \widetilde{H}_t, \widetilde{\sigma}_t)$
  by \cref{homogPresXt},
  and hence, we have (\ref{SLmainQdle}).
\end{proof}

    \section{The conjugate quandle of $\PSL(2, \CC)$}\label{PSL2}

In this section,
we consider the conjugate quandle of $\PSL(2, \CC)$.
Note that the projection 
$\pi: \SL(2, \CC) \to \PSL(2,\CC)$ is 
a covering map both 
as topological groups
and as quandles 
in the sense of \cite{Eisermann-2014-QuandleCoveringsTheirGalois}.
We denote the image $\pi(A)$ of $A \in \SL(2,\CC)$ by $f_A$.
Let $\mathcal{T}$ be the quotient space of $\CC$ modulo $z \sim -z$.
Note that $[t] = \{\pm t\} \in \mathcal{T}$ for each $t \in \CC$.
For $\tau \in \mathcal{T}$, we define 
the subset $X_\tau$ of $\PSL(2, \CC)$,
the base point $f_\tau \in X_\tau$,
and the subgroup $H_\tau$ of $\PSL(2,\CC)$ by 
\begin{align*}
  X_\tau &= 
    \{f_A \in \PSL(2, \CC) \setminus \{ \id \} \mid  \tr A \in \tau \},\\
  f_\tau &=
    \begin{cases*}
      T_1
      & if $\tau = [2]$,
      \\
      S_{\lambda_t^2}
      & if $\tau \in \mathcal{T} \setminus \{[2]\}$,\\
    \end{cases*}\\
  H_\tau &= 
    \begin{cases*}
      \left\{
        T_v \in \PSL(2, \CC)
        \mid 
        v \in \CC
      \right\}
      & if $\tau=[2]$,\\
      \left\{
        S_\lambda \in \PSL(2, \CC)
        \mid 
        \lambda \in \CC^*
      \right\}
      & if $\tau \in \mathcal{T} \setminus \{[0],[2]\}$,\\
      \left\{
        S_\lambda \in \PSL(2, \CC)
        \mid 
        \lambda \in \CC^*
      \right\}
      \cup 
      \left\{
        S_\lambda' \in \PSL(2, \CC)
        \mid 
        \lambda \in \CC^*
      \right\}
      & if $\tau=[0]$.
    \end{cases*}
\end{align*}
Note that $\pi^{-1}(X_\tau) = \bigsqcup_{t \in \tau} \widetilde{X}_t$.
By the classification of isometries,
the elements in $X_\tau$ have the same type.
The right action of $\SL(2, \CC)$
on $\widetilde{X}_t$ 
induces a right action of $\PSL(2, \CC)$ on $X_\tau$ ($\tau = [t]$).
For $\tau =[t] \in \mathcal{T}$ ($t \in \CC$),
the base point satisfies $f_\tau = f_{A_t} = \pi(A_t)$.
The action is transitive
since the action of $\SL(2, \CC)$ on $\widetilde{X}_t$ is transitive.
Note that $g \in H_\tau$ if and only if $g$ and $f_\tau$ commute,
and hence
$H_\tau$ is equal to 
the stabilizer subgroup at $f_\tau$ 
under the action of $\PSL(2, \CC)$.

\begin{lem}\label{X_tauIsCpxMfd}
  $X_\tau$ is a 2-dimensional complex manifold.
\end{lem}
\begin{proof}
  The quandle isomorphism
  $\phi: \widetilde{X}_t \to \widetilde{X}_{-t}$
  satisfies $\pi \circ \phi = \pi$.
  Then
  $X_\tau$ is the orbit space of the action of the group generated by $\iota$
  on $\pi^{-1}(X_\tau) = \bigsqcup_{t \in \tau}\widetilde{X}_t$.
  The action is properly discontinuous since the group is finite.
  Since $\pi^{-1}(X_\tau)$ is 
  a 2-dimensional complex manifold from \cref{X_tIsCpxMfd},
  the orbit space is a complex manifold.
\end{proof}

\begin{prop}\label{homogSpPSL}
  The map $F: H_\tau \backslash \PSL(2, \CC) \to X_\tau$
  defined by $F(H_\tau g) = f_\tau \cdot g$ for $g \in \PSL(2, \CC)$
  is a homeomorphism.
\end{prop}

\begin{proof}
  Apply \cref{homogHomeo} to 
  the transitive action of $\PSL(2,\CC)$ on $X_\tau$.
\end{proof}


We can define a quandle structure on $X_\tau$
as the subquandle of $\Conj(\PSL(2,\CC))$
and denote the quandle by $\Conj(X_\tau)$.
Note that $X_{[2]}$ is equal to the quandle $\mathcal{P}$
in \cite{Inoue-2014-QuandleHomologyComplexVolume}.

\begin{prop}\label{connectedPSL}
  Let $\tau \in \mathcal{T}$.
  \begin{enumerate}
    \item\label{connPSL1} 
    The subset $X_\tau$ 
    generates $\PSL(2,\CC)$ as a group.
    \item\label{connPSL2} 
    The inner automorphism group $\Inn^\Qdle{(\Conj(X_\tau))}$
    is isomorphic to $\PSL(2, \CC)$.
    \item\label{connPSL3} 
    The quandle $\Conj(X_\tau)$ is a connected quandle.
  \end{enumerate}
\end{prop}

\begin{proof}
  (\ref{connPSL1})
  For $\tau=[t] \in \mathcal{T}$,
  since
  $X_\tau$ is the image of $\widetilde{X}_t$ by
  the surjective group homomorphism
  $\pi: \SL(2, \CC) \to \PSL(2, \CC)$
  and \cref{connectedSL} (\ref{connSL1}),
  $X_\tau$ generates $\PSL(2, \CC)$.

  (\ref{connPSL2})
  Since $\langle X_\tau \rangle_{\Grp}$ is equal to $\PSL(2, \CC)$
  by (\ref{connPSL1}), 
  the inner automorphism group $\Inn^\Qdle{(\Conj(X_\tau))}$
  is isomorphic to the inner automorphism group
  $\Inn^\Grp{(\langle X_\tau \rangle_{\Grp})} = \Inn^\Grp(\PSL(2,\CC))$.
  Since the center of $\PSL(2, \CC)$ is trivial,
  the group $\Inn^\Grp(\PSL(2, \CC))$ is isomorphic to $\PSL(2, \CC)$.
  Therefore,
  $\Inn^\Qdle{(\Conj(X_\tau))}$ is isomorphic to $\PSL(2, \CC)$.
  
  (\ref{connPSL3})
  The group isomorphism $S: \PSL(2, \CC) \to \Inn^\Qdle(\Conj(X_\tau))$
  satisfies $f \cdot g = f \lhd S(g)$ 
  for $f \in \Conj(X_\tau)$ and $g \in \PSL(2, \CC)$.
  Since the action of $\PSL(2,\CC)$ is transitive,
  $\Inn{X_\tau}$ acts transitively, which completes of proof. 
\end{proof}

We define an automorphism 
$\sigma_\tau \in \Aut{(\PSL(2, \CC))}$
by the inner automorphism 
$\sigma_\tau = \iota_{f_\tau^{-1}}$ of $f_\tau^{-1}$.
The restriction of the automorphism $\sigma_\tau$ 
to the quandle $X_\tau$ is equal to the point symmetry of $X_\tau$
at the base point $f_\tau$.

\begin{prop}\label{homogPresPSL}
  The quandle $\Conj(X_\tau)$ is isomorphic to 
  $Q(\PSL(2, \CC), H_\tau, \sigma_\tau)$.
\end{prop}

\begin{proof}
  Apply \cref{homogeneousQuandle} to
  the action of $\PSL(2,\CC)$ on $\Conj(X_\tau)$.
\end{proof}

\begin{rmk}
  For $\tau = [t] \in \mathcal{T}$,
  the right action of $\SL(2, \CC)$
  on $\widetilde{X}_t$ 
  induces a right action of $\SL(2, \CC)$ on $X_\tau$
  such that the following diagram is commutative:
  \begin{equation*}
    \xymatrix{
      \SL(2, \CC) \times \widetilde{X}_t \ar[r] \ar[d]^{\id \times \pi} & \widetilde{X}_t \ar[d]^{\pi} \\
      \SL(2, \CC) \times X_\tau \ar[r] & X_\tau
    }
  \end{equation*}
  For $\tau =[t] \neq [0]$,
  the action induces a quandle triplet of $X_\tau$
  that is the same as that of $\widetilde{X}_t$ in \cref{homogPresXt}.
  Therefore,
  $X_\tau$ is isomorphic to $\widetilde{X}_t$ as a quandle
  for $\tau = [t] \neq [0]$.
\end{rmk}

\begin{thm}\label{PSLmain}
  The conjugate quandle $\Conj(\PSL(2, \CC))$ 
  of $\PSL(2,\CC)$ 
  is decomposed into connected components as follows:
  \begin{equation*}
    \Conj(\PSL(2, \CC)) = \{\id\} \sqcup 
    \left(\bigsqcup_{\tau \in \mathcal{T}} \Conj(X_\tau) \right).
  \end{equation*}
  Each connected component $\Conj(X_\tau)$ is
  \begin{enumerate}
    \item\label{PSLmainSpace} 
    a $2$-dimensional complex manifold
    that is homeomorphic to $H_\tau \backslash \PSL(2, \CC)$, and
    
    \item\label{PSLmainQdle}
    a connected subquandle of $\Conj(\PSL(2, \CC))$
    that is isomorphic to $Q(\PSL(2, \CC), H_\tau, \sigma_\tau)$.
  \end{enumerate}
\end{thm}

\begin{proof}
  Since we have the decomposition of $\Conj(\SL(2, \CC))$ in \cref{SLmain}
  and the surjective quandle homomorphism 
  $\pi: \Conj(\SL(2, \CC)) \to \Conj(\PSL(2, \CC))$
  satisfies 
  $\pi^{-1}(\id) = \{I_2\} \sqcup \{-I_2\}$
  and $\pi^{-1}(X_\tau) = \bigsqcup_{t \in \tau} \widetilde{X}_t$,
  we have the decomposition of $\Conj(\PSL(2, \CC))$
  and each component is a connected component.
  
  The topological space
  $\Conj(X_\tau)=X_\tau$
  is a $2$-dimensional complex manifold by \cref{X_tauIsCpxMfd}
  and homeomorphic to $H_\tau \backslash \PSL(2, \CC)$
  by \cref{homogSpPSL}.
  Hence, we have (\ref{PSLmainSpace}).
  $\Conj(X_\tau)$ is a connected quandle by \cref{connectedPSL}
  and isomorphic to $Q(\PSL(2, \CC), H_\tau, \sigma_\tau)$
  by \cref{homogPresPSL}.
  Hence we have (\ref{PSLmainQdle}).
\end{proof}

    \section{The quandles obtained from Kleinian groups}\label{Klein}

In this section, 
we define a quandle $Q(\Gamma, \gamma)$
and show the main theorem.
Let $\Gamma$ be a Kleinian group,
$\gamma$ a non-trivial element in $\Gamma$,
and $\sigma_\gamma$ the inner automorphism $\iota_{\gamma^{-1}}$ of $\gamma^{-1}$.
We denote the centralizer of $\sigma_\gamma$ by $C_\Gamma(\gamma)$.
\begin{dfn}\label{KleinQdle}
  For a Kleinian group $\Gamma$ and 
  a non-trivial element $\gamma$ in $\Gamma$,
  let $Q(\Gamma, \gamma)$
  be the quandle defined 
  by the quandle triplet $(\Gamma , C_\Gamma(\gamma), \sigma_\gamma)$.
\end{dfn}

We can regard $Q(\Gamma, \gamma)$ as 
a generalization of the knot quandle 
of a hyperbolic knot 
by the next proposition.

\begin{prop}\label{hypKnotQdle}
  If $\Gamma$ is isomorphic to 
  the knot group $G(K)$ of a hyperbolic knot $K$
  and 
  $\gamma \in \Gamma$ is a parabolic element
  corresponding to the meridian $m \in G(K)$,
  then $Q(\Gamma, \gamma)$ is isomorphic to 
  the knot quandle $Q(K)$ of $K$.
\end{prop}

\begin{proof}
  Since $K$ is a hyperbolic knot,
  there exists a holonomy representation
  $\rho: G(K) \to \PSL(2, \CC)$
  such that $\rho(G(K)) = \Gamma$ and $\rho(m) = \gamma$.
  We identify the knot group $G(K)$ 
  and the Kleinian group $\Gamma$ by the representation $\rho$.
  Then, the group $C_\Gamma(\gamma)$ is equal to 
  the peripheral subgroup $P(K)$ of $G(K)$.
  The knot quandle $Q(K)$ is isomorphic to the quandle $Q(G(K), P(K), \sigma_m)$,
  where $\sigma_m$ is the inner automorphism $\iota_{m^{-1}}$
  \cite[Corollary 16.2]{Joyce-1982-ClassifyingInvariantKnotsKnot}.
  Hence, $Q(\Gamma, \gamma)$ and $Q(K)$ are obtained 
  from the same quandle triplet 
  $(\Gamma, C_\Gamma(\gamma), \sigma_\gamma)$ and $(G(K), P(K), \sigma_m)$,
  which completes the proof.
\end{proof}


\begin{lem}\label{ConjIso}
  The map 
  $\overline{\iota_f}: Q(\Gamma, \gamma) 
  \to Q(\iota_{f}(\Gamma), \iota_f(\gamma))$
  induced by $\iota_f$
  is a quandle isomorphism,
  where $\iota_f$ is the inner automorphism induced by $f \in \PSL(2, \CC)$.
\end{lem}

\begin{proof}
  Since
  any $g \in \Gamma$ satisfies 
  \begin{align*}
    \sigma_{\iota_f(\gamma)} \circ \iota_f(g) 
    = (f \gamma f^{-1})^{-1} (f g f^{-1}) (f \gamma f^{-1})
    = f (\gamma^{-1} g \gamma) f^{-1}
    = \iota_f \circ \sigma_\gamma (g),
  \end{align*}
  we have $\sigma_{\iota_f(\gamma)} \circ \iota_f = \iota_f \circ \sigma_\gamma$.
  By definition,
  we have $\iota_f(C_\Gamma(\gamma)) = C_{\iota_f(\Gamma)}(\iota_f(\gamma))$.
  Since $\iota_f$ is bijective,
  we have $C_\Gamma(\gamma) = \iota_f^{-1}(C_{\iota_f(\Gamma)}(\iota_f(\gamma)))$.
  Therefore, 
  $\overline{\iota_f}: Q(\Gamma, \gamma) \to Q(\iota_f(\Gamma), \iota_f(\gamma))$
  is an injective quandle homomorphism by \cref{injective_lemma}.
  Since $\iota_f$ is surjective, 
  the map $\overline{\iota_f}$ is surjective.
  Therefore, $\iota_f$ is a quandle isomorphism.
\end{proof}

\begin{lem}\label{inducedHom}
  If $\gamma \in H_\tau$,
  then
  there exists 
  an injective quandle homomorphism 
  $\psi: Q(\Gamma, \gamma) \to Q(\PSL(2, \CC), H_\tau, \sigma_\tau)$.
\end{lem}

\begin{proof}
  There exists $f \in \PSL(2, \CC)$
  such that $f \gamma f^{-1} = f_\tau$,
  where $f_\tau$ is the base point of $X_\tau$.
  The map $\overline{\iota_f}: Q(\Gamma, \gamma) \to Q(\Gamma^f, f^{-1} \gamma f)$
  is a quandle isomorphism by \cref{ConjIso}.
  Hence, we can assume $\Gamma = \iota_f(\Gamma)$ and $\gamma = f_\tau$.
  Then the group $C_\Gamma(\gamma)$ is a subgroup of $H_{\tau}$
  and $\sigma_{\gamma} = \sigma_\tau$.
  Hence,
  the inclusion map $i: \Gamma \to \PSL(2, \CC)$
  induces the quandle homomorphism 
  $\overline{i}: Q(\Gamma, \gamma) \to Q(\PSL(2, \CC), H_\tau, \sigma_\tau)$
  by \cref{injective_lemma} (\ref{inj1}).
  Since $C_\Gamma(\gamma) = H_\tau \cap \Gamma = i^{-1}(H_\tau)$,
  the quandle homomorphism $\overline{i}$ is injective 
  by \cref{injective_lemma} (\ref{inj2}).
  Therefore, $\psi := \overline{i} \circ \overline{\iota_f}$
  is an injective quandle homomorphism.
\end{proof}

\begin{dfn}\label{canonicalmap}
  Let $\psi: Q(\Gamma, \gamma) \to \Conj(\PSL(2, \CC))$
  be the injective quandle homomorphism in \cref{inducedHom},
  $F: Q(\PSL(2, \CC), H_\tau, \sigma_\tau) \to \Conj(X_\tau)$
  the quandle isomorphism in \cref{PSLmain}, 
  and $i: \Conj(X_\tau) \to \Conj(\PSL(2, \CC))$
  the inclusion map.
  We call the composite map 
  $i \circ F \circ \psi: Q(\Gamma, \gamma) \to \Conj(\PSL(2, \CC))$
  a \emph{canonical map}.
\end{dfn}

The proof of the discreteness of $Q(\Gamma, \gamma)$ 
is reduced to the following lemma.

\begin{lem}[
  {\cite[Section 1.2, Lemma 4]{Godement-2017-IntroductionTheoryLieGroups}}
  ]
  \label{discrete_lemma}
  Let $G$ be a topological group,
  $H$ and $K$ closed subgroups of $G$,
  and $p: G \to G/H$ the natural projection.
  Suppose that $G$ is locally compact and Hausdorff,
  and $H / (H \cap K)$ is compact.
  Then the following hold:
  \begin{enumerate}
    \item The image $p(K)$ is closed in $G/H$.
    \item If $K$ is discrete, then $p(K)$ is discrete in $G/H$.
  \end{enumerate}
\end{lem}


\begin{lem}\label{elemPara}
  If $\gamma$ is parabolic, 
  then $C_\Gamma(\gamma)$ is an elementary group of parabolic type.
\end{lem}
\begin{proof}
  Suppose that $\gamma \in \Gamma$ is parabolic.
  Then there uniquely exists $x \in \partial \HH^3$
  such that $\gamma(x) = x$.
  Any $g \in C_\Gamma(\gamma)$ satisfies
  \begin{equation*}
    g(x) = g(\gamma (x)) = \gamma(g(x)).
  \end{equation*}
  Hence $g(x)$ is a fixed point of the parabolic element $\gamma$,
  and we have $g(x) = x$.
  Therefore the orbit of $x$ by $C_\Gamma(\gamma)$ is finite
  and we have that $C_\Gamma(\gamma)$ is elementary.
  Since $C_\Gamma(\gamma)$ is discrete in $\PSL(2, \CC)$ 
  and has a parabolic element,
  $C_\Gamma(\gamma)$ has no loxodromic element 
  \cite[Theorem 5.5.4]{Ratcliffe-2019-FoundationsHyperbolicManifolds}.
  Hence $C_\Gamma(\gamma)$ preserves horospheres based at $x$,
  thus $C_\Gamma(\gamma)$ is conjugate to a discrete subgroup of $\Isom^+(\EE^2)$
  in $\PSL(2,\CC)$.
  Therefore,
  $C_\Gamma(\gamma)$ is an elementary group of 
  parabolic type by \cref{discElemPara}.
\end{proof}

\begin{prop}\label{discretePara}
  Suppose that a Kleinian group $\Gamma$ is isomorphic 
  to the fundamental group of 
  a complete hyperbolic $3$-orbifold of finite volume.
  Let $\gamma$ be a parabolic element in $\Gamma$.
  Then,
  the image of the map 
  $\psi: Q(\Gamma, \gamma) \to Q(\PSL(2, \CC), H_{[2]}, \sigma_{[2]})$ 
  in \cref{inducedHom}
  is discrete in $H_{[2]} \backslash \PSL(2, \CC)$.
\end{prop}

\begin{proof}
  We can assume $\gamma = f_{[2]} = T_1 \in H_{\tau}$ 
  by \cref{ConjIso}.
  The group $C_\Gamma(\gamma)$ is a discrete elementary group of parabolic type
  by \cref{elemPara}.
  Since $\HH^3/\Gamma$ has finite volume,
  there exists a finite index subgroup $G$ of $C_\Gamma(\gamma)$
  such that $G$ is isomorphic to $\ZZ^2$.
  Hence $H_{[2]}/G \cong \CC/\ZZ^2$ is homeomorphic to the 2-torus.
  Since $H_{[2]}/G$ is a finite-sheeted covering space of $H_{[2]} / C_\Gamma(\gamma)$,
  we have that $H_{[2]} / C_\Gamma(\gamma)$ is compact.
  Therefore, we can apply \cref{discrete_lemma}
  with $G=\PSL(2, \CC)$, $H = H_{[2]}$ and $K = \Gamma$,
  which completes of the proof.
\end{proof}

\begin{lem}\label{elemNotPara} 
  If $\gamma$ is not parabolic 
    and the cardinality of $C_\Gamma(\gamma)$ is infinite, 
    then $C_\Gamma(\gamma)$ is an elementary group of hyperbolic type.
\end{lem}
\begin{proof}
  Suppose that $\gamma$ is not parabolic.
  From the classification of transformations,
  there uniquely exists a geodesic $l$ in $\HH^3$
  such that the restriction of $\gamma$ 
  to the set $\partial_\infty (l)$ 
  is the identity map.
  Note that the cardinality of $\partial_\infty(l)$ is $2$. 
  Then
  any $g \in C_\Gamma(\gamma)$ and any $x \in \partial_\infty (l)$
  satisfy
  \begin{equation}
    g(x) = g(\gamma(x)) = \gamma(g(x)).
  \end{equation}
  Hence $g(x)$ is a fixed point of 
  $\gamma|_{\partial_\infty(l)}$,
  and we have $g(x) \in \partial_\infty(l)$.
  Hence
  for any $x \in \partial_\infty (l)$,
  $C_\Gamma(\gamma) \cdot x \subset \partial_\infty (l)$.
  Therefore
  $C_\Gamma(\gamma)$ is elementary.
  Assume that the cardinality of $C_\Gamma(\gamma)$ is infinite.
  Then $C_\Gamma(\gamma)$ is not of elliptic type
  by \cref{discElemEllipt}.
  Let $\partial_\infty(l) =\{x,y\}$.
  Then we have $C_\Gamma(\gamma) \cdot x \subset \partial_\infty(l)$.
  If $C_\Gamma(\gamma) \cdot x = \partial_\infty(l)$,
  then $C_\Gamma(\gamma)$ is not parabolic
  since $\# C_\Gamma(\gamma) \cdot x = 2$.
  If $C_\Gamma(\gamma) \cdot x = \{x\}$,
  then $C_\Gamma(\gamma)$ is not parabolic
  since $C_\Gamma(\gamma) \cdot y$ is another finite orbit.
  Therefore
  $C_\Gamma(\gamma)$ is an elementary group of hyperbolic type,
  which completes of the proof.
\end{proof}

\begin{prop}\label{discreteNotPara}
  Let $\Gamma$ be a Kleinian group
  and
  $\gamma$ be a non-trivial element in $\Gamma$.
  Suppose that $\gamma$ is not parabolic
  and the cardinality of $C_\Gamma(\gamma)$ is infinite.
  Then
  the image of the map 
  $\psi: Q(\Gamma, \gamma) \to Q(\PSL(2, \CC), H_\tau, \sigma_\tau)$ 
  in \cref*{inducedHom}
  is discrete in $H_\tau \backslash \PSL(2, \CC)$.
\end{prop}

\begin{proof}
  We can assume 
  $\gamma = f_{\tau} \in H_{\tau}$
  for some $\tau \in \mathcal{T} \setminus \{[2]\}$
  by \cref{ConjIso}.
  The subgroup $C_\Gamma(\gamma)$ is an elementary discrete subgroup 
  of hyperbolic type according to \cref{elemNotPara}.
  Therefore,
  there exists a finite index subgroup $G$ of $C_\Gamma(\gamma)$
  generated by a loxodromic transformation $g \in C_\Gamma(\gamma)$
  as stated in \cref{discElemHyp}.
  Since $C_\Gamma(\gamma)$ is contained in $H_\tau$,
  we can assume that
  $g = S_{\mu}$ for $\mu \in \CC$ with $|\mu| > 1$.
  Then $H_\tau / G$ is compact.
  Indeed,
  $H_\tau / G$ is 
  homeomorphic to the 2-torus $\CC^*/ \langle S_\mu \rangle_\Grp$
  if $\tau \neq [0]$,
  or is homeomorphic to the disjoint union
  $(\CC^*/ \langle S_\mu \rangle_\Grp) \sqcup 
  (\CC^*/ \langle S_\mu \rangle_\Grp)$ of two of the 2-tori
  if $\tau = [0]$.
  Since $G$ is a finite index subgroup in $C_\Gamma(\gamma)$,
  the space $H_\tau / G$ is 
  a finite-sheeted covering space of $H_\tau / C_\Gamma(\gamma)$,
  thus $H_\tau / C_\Gamma(\gamma)$ is compact.
  Therefore we can apply \cref{discrete_lemma}
  by $G=\PSL(2, \CC)$, $H = H_{\tau}$ and $K = \Gamma$,
  which is complete of the proof.
\end{proof}

Finally, we show the main theorem.

\begin{proof}[Proof of \cref{mainthm}]
  We use the notation in \cref{canonicalmap}.
  There exists $\tau \in \mathcal{T}$ such that $\gamma \in X_\tau$
  since $\gamma$ is a non-trivial element.
  By \cref{inducedHom},
  the map $\psi: Q(\Gamma, \gamma) \to Q(\PSL(2, \CC), H_\tau, \sigma_\tau)$
  is an injective quandle homomorphism.
  Since $F: Q(\PSL(2, \CC), H_\tau, \sigma_\tau) \to \Conj(X_\tau)$
  is a quandle isomorphism
  and 
  the inclusion map $i: \Conj(X_\tau) \to \Conj(\PSL(2, \CC))$
  is an injective quandle homomorphism,
  the canonical map $i \circ F \circ \psi: Q(\Gamma, \gamma) \to \Conj(\PSL(2, \CC))$ 
  is an injective quandle homomorphism.

  Suppose that $\Gamma$ and $\gamma$ satisfy one of the assumptions.
  From \cref{discretePara} or \cref{discreteNotPara},
  the image of $\psi$ is discrete in $H_\tau \backslash \PSL(2, \CC)$.
  Since $F$ is a homeomorphism,
  the image of the map $F \circ \psi$ is discrete in $X_\tau$.
  Since $X_\tau$ is endowed with the relative topology from $\PSL(2, \CC)$,
  the image of the canonical map $i \circ F \circ \psi$
  is discrete in $\Conj(\PSL(2, \CC))$,
  which completes of the proof.
\end{proof}

    \section*{Acknowledgement}
    This work was supported by JST, the establishment of university 
    fellowships towards the creation of science technology innovation, 
    Grant Number JPMJFS2138.
    \bibliographystyle{plain}
    \bibliography{reference}
    \address{
        (R. Kai) 
        Department of Mathematics, 
        Graduate School of Science, 
        Osaka Metropolitan University, 
        3-3-138, Sugimoto, 
        Sumiyoshi-ku, Osaka, 558-8585, Japan}  
    \email{sw23889b@st.omu.ac.jp}

\end{document}